\documentclass[leqno]{amsart}
\usepackage{amsmath,amsbsy,amsfonts}
\usepackage{color}
\usepackage{enumerate}
\date{}
\def\dis{\displaystyle}
\def\nd{\noindent}
\def\thend{\rule{3mm}{3mm}}
\def\Re{\mathbb{R}}
\newtheorem{thm}{Theorem}[section]
\newtheorem{cor}[thm]{Corollary}
\newtheorem{lem}[thm]{Lemma}
\newtheorem{prop}[thm]{Proposition}
\newtheorem{rem}[thm]{Remark}
\newtheorem{exm}[thm]{Example}

\begin{document}

\title[Elliptic equations involving the $p$-Laplacian]{Elliptic equations involving the $p$-Laplacian and a gradient term having natural growth}
\vspace{1cm}
%%%%%%%%%%%%%%%%%%%%%%%%%%%%%%%%%%%%%%%%%%%%%%%%%%%%%%%%%%%%%%%%%%%%%%%%

\author{D. G. de Figueiredo}
\address{D. G. de Figueiredo \newline IMECC-UNICAMP, Caixa Postal 6065, Campinas-SP 13083-859, Brazil}
\email{djairo@ime.unicamp.br}

\author{J-P. Gossez}
\address{J.-P. Gossez \newline D\'epartement de Math\'ematique, C.P. 214, Universit\'e Libre de Bruxelles, 1050 Bruxelles, Belgium}
\email{gossez@ulb.ac.be}

\author{H. Ramos Quoirin}
\address{H. Ramos Quoirin \newline Universidad de Santiago de Chile, Casilla 307, Correo 2, Santiago, Chile}
\email{\tt humberto.ramos@usach.cl}

\author{P. Ubilla}
\address{P. Ubilla \newline Universidad de Santiago de Chile, Casilla 307, Correo 2, Santiago, Chile}
\email{\tt pedro.ubilla@usach.cl}


\thanks{D.G de Figueiredo was supported by CNPq. J-P. Gossez was supported by FNRS. H. Ramos Quoirin and P. Ubilla were supported by Fondecyt 1161635.}

\begin{abstract}
We investigate  the problem $$
\left\{
\begin{array}{ll}
-\Delta_p u = g(u)|\nabla u|^p +  f(x,u) \
& \mbox{in} \ \ \Omega, \ \ \\
u>0 \
&\mbox{in} \ \ \Omega, \ \ \\
 u = 0 \ &\mbox{on} \ \
\partial\Omega,
\end{array}
 \right. \leqno{(P)}
$$
in a bounded smooth domain $\Omega \subset \mathbb{R}^N$. Using a Kazdan-Kramer change of variable we reduce this problem to a quasilinear one without gradient term and therefore approachable by variational methods. In this way we come to some new and interesting problems for quasilinear elliptic equations which are motivated by the need to solve $(P)$. Among other results, we investigate the validity of the Ambrosetti-Rabinowitz condition according to the behavior of $g$ and $f$. Existence and multiplicity results for $(P)$ are established in several situations.
\end{abstract}

\maketitle

\section{\textbf{Introduction}}
\bigskip

This paper is concerned with existence, non-existence and multiplicity of solutions for the problem

$$
\left\{
\begin{array}{ll}
-\Delta_p u = g(u)|\nabla u|^p +  f(x,u) \
& \mbox{in} \ \ \Omega, \ \ \\
u>0 \
&\mbox{in} \ \ \Omega, \ \ \\
 u = 0 \ &\mbox{on} \ \
\partial\Omega.
\end{array}
 \right.\leqno{(P)}
$$
Here $\Delta_p u:= \text{div} \left(|\nabla u|^{p-2} \nabla u \right)$ is the $p$-Laplacian operator with $1<p<\infty$ and
 $\Omega\subset\mathbb{R}^{N}$ is a smooth bounded domain.
 
Equations like $(P)$ have attracted a considerable interest since the well-known works of Kazdan-Kramer \cite{KK} and Serrin  \cite{Se}. In \cite{Se} it was observed, also for $p=2$, that the quadratic growth with respect to the gradient plays a critical role as far as existence is concerned. Since then,  a large literature has been devoted to problems where the power in the gradient is strictly less than $p$. We shall focus here on the so-called natural growth of the gradient for the $p$-Laplacian, which is given precisely by  $|\nabla u|^p$.  
In \cite{KK} it was observed,  for $p=2$,  that the equation in $(P)$ enjoys some invariance property and can be transformed through a suitable change of variable into an equation without  gradient term. Although the transformed problem has no gradient term, variational methods do not apply in a straightforward way. The difficulty lies in establishing a compactness condition (e.g. the Palais-Smale condition) for the functional associated to the differential equation. Many authors have studied $(P)$ by different methods like degree theory, sub-super solutions, a priori estimates, etc. We refer, for instance, to \cite{BoMuPu,BEZF,FaMiMoTa,ILS,LiYenKe,MoMo,MoMu,PoSe}. 
Some special cases have been studied by using variational arguments and we refer, for instance, to \cite{ADP, GMIRQ, JRQ, JS} for the case $p=2$ and \cite{HBV,ILU} for $p>1$.
 
 One of difficulties in the variational approach related to $(P)$ is the fact that in many cases the nonlinearity in the transformed equation does not satisfy the Ambrosetti-Rabinowitz condition. For instance, when $g \equiv 1$, the nonlinearity in the transformed equation (see below) satisfies the Ambrosetti-Rabinowitz condition only if  $F(x,s):=\int_0^s f(x,t) dt$ has at least an exponential growth (see Remark \ref{rexp}). There are however a certain number of situations where it does satisfy this condition. One of our purposes in the present paper is to investigate these situations in a rather systematic way.
 
 %To handle the situation where $F(x,s)$ does not have an exponential growth, we use a monotonicity type condition.
 
We will thus follow the approach initiated for $p=2$ in \cite{KK}, transforming $(P)$ into a problem of the form 
$$
\left\{
\begin{array}{ll}
-\Delta_p v = h(x,v) \
& \mbox{in} \ \ \Omega, \ \ \\
v>0 \
&\mbox{in} \ \ \Omega, \ \ \\
 v = 0 \ &\mbox{on} \ \
\partial\Omega.
\end{array}
 \right.\leqno{(Q)}
$$
The suitable change of variables $v=A(u)$ turns out to be
\begin{equation}
A(s):=\int_0^s e^{\frac{G(t)}{p-1}} \ dt,
\end{equation} 
where
$G(s)=\int_0^s g(t) \ dt$ and
\begin{equation}
h(x,s):=e^{G(A^{-1}(s))} f(x,A^{-1}(s)).
\end{equation}
Details are given at the beginning of Section 3.

 We will consider two cases: $h$ $p$-superlinear at zero (Subsection 2.1) and $h$ $p$-sublinear at zero (Subsection 2.2). As we will see in Lemma 5.3, this classification corresponds to $f$ being $p$-superlinear at zero and $f$ being $p$-sublinear at zero, respectively.

In our first result (Theorem \ref{t1}), $h(x,s)$ is $p$-superlinear at zero and satisfies the Ambrosetti-Rabinowitz condition. Here are a few equations which can be handled by Theorem \ref{t1}:\\
\begin{enumerate}
\item[(i)] $-\Delta_p u= C_1 |\nabla u|^p +u^q e^{C_2 u}$, where $C_1>0$, $0<C_2<C_1 \frac{p^*-p}{p-1}$ and $q>p-1$ (cf. Example \ref{e1});\\

\item[(ii)] $-\Delta_p u= C (1+u)^{-\alpha} |\nabla u|^p +\mu u^{p-1} e^{\beta u^{1-\alpha}}$, where $C>0$, $0<\alpha<1$, $0<\beta <\frac{p^*-p}{(p-1)(1-\alpha)}$ and $0<\mu<\lambda_1$  (cf. Example \ref{e2});\\

\item[(iii)] $-\Delta_p u= C (1+u)^{-\alpha} |\nabla u|^p +u^{r-1}$, where $C>0$, $\alpha \geq 1$, $p<r<p^*$ and, in addition, $C<r-p$ if $\alpha=1$  (cf. Example \ref{e3});\\

\item[(iv)] $-\Delta_p u= u^q |\nabla u|^p +\mu u^{p-1} e^{\beta u^{q+1}}$, where $q>0$, $0<\beta<\frac{p^*-p}{(p-1)(q+1)} $ and $0<\mu<\lambda_1$  (cf. Example \ref{e4}).\\

\end{enumerate}

We have used above the standard notation: $p^*$ is the Sobolev conjugate exponent given by $\frac{1}{p^*}=\frac{1}{p}-\frac{1}{N}$ with $p^*=\infty$ when $p\geq N$, and $\lambda_1$ is the first eigenvalue of $-\Delta_p$ in $W_0^{1,p}(\Omega)$.

Our second result (Theorem \ref{t2}) concerns once again the case where $h(x,s)$ is $p$-superlinear at zero but does not necessarily satisfy the Ambrosetti-Rabinowitz condition.  The approach here relies instead on a monotonicity condition which enables the verification of the Cerami condition. This monotonicity condition has been used recently by Liu \cite{Liu} and Iturriaga-Lorca-Ubilla \cite{ILU} (see also Miyagaki-Souto \cite{Mi-So} for $p=2$). Here are a few equations which can be handled by Theorem \ref{t2}:\\

\begin{enumerate}
\item[(v)] $-\Delta_p u=\frac{p-1}{u+1} |\nabla u|^p +u^{p-1}(\log(u+1))^q$, where $q>0$  (cf. Example \ref{e5});\\

\item[(vi)] $-\Delta_p u= C|\nabla u|^p +u^{r-1}$, where $C>0$, and $p<r$ (cf. Example \ref{e6});\\

\item[(vii)] $-\Delta_p u= C|\nabla u|^p +(\log(u+1))^{r-1}$, where $C>0$, and $p<r$ (cf. Example \ref{e6}).\\ 

\end{enumerate}

We remark, in equation (vi), that if $C=0$ and $r= p^*-1$ then $(P)$ reduces to the Pohozaev problem, which has no solution when $\Omega$ is starshaped. In the case $p=2$, the famous result of Brezis-Nirenberg \cite{BrNi} states that the existence of a solution can be recovered if one adds a perturbation such as $\lambda u$. This result was generalized to the case $1<p^2\leq N$ by Garcia-Peral \cite{GaPe}, Egnell \cite{Eg} and Guedda-Veron \cite{GueVe}. It follows from our Theorem \ref{t2} that the existence of a solution can also be recovered in the spirit of Brezis-Nirenberg  with a perturbation such as $C|\nabla u|^p$.

It is also worthwhile to compare examples (i) and (vi), which differ only by the presence of an exponential term: the Ambrosetti-Rabinowitz condition holds for (i), but not for (vi).

Theorems \ref{t1} and \ref{t2} are stated in Subsection 2.1.

Our third result concerns the case where $h(x,s)$ is $p$-sublinear at zero and is stated in Subsection 2.2.  We introduce a parameter $\lambda>0$ in $(P)$:
$$
\left\{
\begin{array}{ll}
-\Delta_p u = g(u)|\nabla u|^p + \lambda f(x,u) \
& \mbox{in} \ \ \Omega, \ \ \\
u>0 \
&\mbox{in} \ \ \Omega, \ \ \\
 u = 0 \ &\mbox{on} \ \
\partial\Omega.
\end{array}
 \right.\leqno{(P_\lambda)}
$$

The transformed problem thus reads
$$
\left\{
\begin{array}{ll}
-\Delta_p v = \lambda h(x,v) \
& \mbox{in} \ \ \Omega, \ \ \\
v>0 \
&\mbox{in} \ \ \Omega, \ \ \\
 v = 0 \ &\mbox{on} \ \
\partial\Omega.
\end{array}
 \right.\leqno{(Q_\lambda)}
$$

 For a nonlinearity of concave-convex type we obtain for $(P_\lambda)$ a result in the line of the classical one by Ambrosetti-Brezis-Cerami \cite{ABC}: for some $0<\Lambda<\infty$, $(P_\lambda)$ admits at least two solutions for $0<\lambda<\Lambda$, at least one solution for $\lambda=\Lambda$, and no solution for $\lambda>\Lambda$. Precise statements are given in Theorems \ref{t3} and \ref{t4}, which are obtained by applying some of the results of \cite{DGU3}. Here are a few equations which can be handled by Theorems \ref{t3} and \ref{t4}:\\

\begin{enumerate}
\item[(viii)] $-\Delta_p u= C_1 |\nabla u|^p + \lambda u^q e^{C_2 u}$, where $C_1>0$, $0<C_2<C_1 \frac{p^*-p}{p-1}$ and $0\leq q < p-1$;\\

\item[(ix)]  $-\Delta_p u= C (1+u)^{-\alpha} |\nabla u|^p +\lambda\left(u^{r-1}+ u^{q-1}\right)$, where $C>0$, $\alpha \geq 1$, $1<q<p<r<p^*$ and, in addition, $C<r-p$ if $\alpha=1$;\\
\end{enumerate}

It is valuable to compare examples (i), (iii) with examples (viii), (ix), respectively: at least one solution for the first ones, at least two solutions for the second ones.

%Section 5 deals with a situation where $g(u)$ in $(P_\lambda)$ is replaced by $g(x,u)$. In this case, a change of variable a la Kazdan-Kramer as above is not possible anymore. We use here...

To conclude this introduction, let us comment on some related works.
The change of variables introduced for $p=2$ by Kazdan-Kramer \cite{KK} has been used, for instance, by Montenegro-Montenegro \cite{MoMo}, Iturriaga- Lorca-Ubilla \cite{ILU}, and Iturriaga-Lorca-S\'anchez \cite{ILS}. In \cite{MoMo} the authors obtain existence of solutions for some specific functions $g$, for instance $g$ constant or $g$ such that $\displaystyle \lim_{s \to \infty} g(s)=0$, cf. Examples 2.3, 3.1 and 4.1 in \cite{MoMo}. In \cite{ILU} the authors prove the existence of a solution when
$g$ is a constant and $f(x,s)$ has at most an exponential growth\,(compare with our example (i)). In \cite{ILS} the authors assume that $g$ is a constant and the model for $f(x,s)$ is a power\,(compare with our example (vii)).

Some situations where it is not possible to use the  change of variables of Kazdan-Kramer have also been considered, for instance, by Li-Yen-Ke \cite{LiYenKe} (see also the references therein), where a similar problem to $(P)$ is studied. However, the function $f$ may also depend  on the gradient.  They consider the case $g(s)=\frac{c}{s+1}$, with a power as a typical model for $f$ (this is related with our example (iii)). 

Let us finally observe that the results in this paper are new even in the case $p=2$.

\medskip
 \section{\textbf{Statements of results}}
\medskip
 
 Throughout this paper, the functions $f$ and $g$ are assumed to satisfy the following conditions:
 
 \begin{itemize}
\item [$(H_g)$] $g:[0,\infty) \rightarrow [0,\infty)$ is continuous.
\item  [$(H_f)$] $f:\Omega \times [0,\infty) \rightarrow [0,\infty)$ is a Carath\'eodory function such that $f(x,s)$ remains bounded when $s$ remains bounded.% and for some $s_0 \geq 0$ and $\varepsilon>0$ we have $f(x,s) \geq \varepsilon$ for {\it a.e.} $x \in \Omega$ and all $s \geq s_0$. Moreover, for some $s_0 \geq 0$ the derivative $\frac{\partial f}{\partial s}(x,s)$ exists for {\it a.e.} $x \in \Omega$ and all $s \geq s_0$. This derivative will be denoted by $f'(x,s)$.
\end{itemize}

By a solution of $(P)$ we mean $u \in \mathcal{C}^1(\overline{\Omega})$ such that $u>0$ in $\Omega$, $u=0$ on $\partial \Omega$, and $u$ satisfies the equation in $(P)$ in the weak sense.

Our first result involves the following four assumptions on $g$ and $f$:\\

 \begin{itemize}
\item [$(H_{SC})$] There exists $r<p^*$ such that $$ \lim_{s \to \infty} \frac{ f(x,s)e^{G(s)}}{\left( \int_0^s e^{\frac{G(t)}{p-1}} dt\right)^{r-1}}=0  $$
uniformly with respect to $x \in \Omega$.\\
\item  [$(H_{AR})_1$] There exist $s_0 \geq 0$ and $\theta>p$ such that $$ 
\theta \int_0^s e^{\frac{p}{p-1}G(t)} f(x,t) dt \leq e^{G(s)} f(x,s)\int_0^s e^{\frac{G(t)}{p-1}}  dt$$
for {\it a.e.} $x \in \Omega$ and all $s \geq s_0$.\\

\item  [$(H_{AR})_2$] There exist $s_0 \geq 0$, $\delta>0$ and a non-empty subdomain $\Omega_1 \subset \Omega$ such that $$ 
\int_0^s e^{\frac{p}{p-1}G(t)}  f(x,t) dt \geq \delta$$
for {\it a.e.} $x \in \Omega_1$ and all $s \geq s_0$.\\

\end{itemize}

As will be seen in the next section, $(H_{SC})$ corresponds to a subcritical growth condition for the transformed problem $(Q)$ (cf. Lemma \ref{l1}),  $(H_{AR})_1$ and $(H_{AR})_2$ correspond to the Ambrosetti-Rabinowitz condition for $(Q)$ (cf. Lemma \ref{l2}). Note that $(H_{SC})$, $(H_{AR})_1$ and $(H_{AR})_2$ concern the behavior of $f$ at infinity. Note also that if $g \equiv 0$, then $(H_{SC})$ reduces to the standard subcritical growth condition for $f(x,s)$ while $(H_{AR})_1$ and $(H_{AR})_2$ reduce to the standard Ambrosetti-Rabinowitz condition for $f(x,s)$.

\medskip
\subsection{{\bf The case $f(x,s)$ $p$-superlinear at zero}}
\strut
\medskip

We recall that $f(x,s)$ is $p$-superlinear at zero when it satisfies the following condition:

\begin{itemize}
\item [$(H_{\lambda_1})$] $\displaystyle \limsup_{s \to 0} \frac{f(x,s)}{s^{p-1}} <\lambda_1$ uniformly with respect to $x \in \Omega$.\\
\end{itemize}

We shall see that under this condition $h(x,s)$ is $p$-superlinear at zero as well (cf. Lemma \ref{l3}). Our first result in this case is:

 \begin{thm}[Existence with the Ambrosetti-Rabinowitz condition] 
 \label{t1}
Assume  $(H_{SC})$, $(H_{AR})_1$, $(H_{AR})_2$ and $(H_{\lambda_1})$. Then problem $(P)$ has at least one solution.
 \end{thm}
 
 Theorem \ref{t1} will be proved in Section 5, and its corollaries in Section 4.
 
%\subsection{}

 To illustrate Theorem \ref{t1}, let us indicate a few typical situations where it applies. We will distinguish a number of cases accordingly to the behaviour of $g$ at infinity. The corollaries and examples below provide various concrete situations where problem $(P)$ can be handled in a variational way, with the standard Ambrosetti-Rabinowitz condition being satisfied.
 
To simplify our statements, we assume a differentiability  condition on $f$:\\
 
 \begin{itemize}
\item [$(H_f)'$] There exist $s_0 \geq 0$ and $\varepsilon>0$ such that $ f(x,s)\geq \varepsilon$ and $\frac{\partial f}{\partial s}(x,s)$ exists for 
for {\it a.e.} $x \in \Omega$ and all $s \geq s_0$. This derivative will be denoted by $f'(x,s)$.\\
\end{itemize}

\begin{cor}\label{c1}
Assume $g(s) \to g_{\infty}$ as $s \to \infty$ with $0<g_\infty<\infty$, as well as $(H_f)'$ and $(H_{\lambda_1})$. Then $(P)$ has at least one solution  if, for some $p<r<p^*$, $$
\lim_{s \to \infty} \frac{f(x,s)}{e^{\frac{r-p}{p-1} G(s)}} =0 \quad \mbox{and} \quad \lim_{s \to \infty} \frac{f'(x,s)}{f(x,s)} >0,
$$
uniformly with respect to $x \in \Omega$.
\end{cor} 
 
\begin{exm}
\label{e1}
{\rm Corollary \ref{c1} applies, for instance, to $g(s) \equiv C_1$ and $f(s)=s^q e^{C_2 s}$ with $C_1>0$, $0<C_2< C_1 \frac{p^*-p}{p-1}$ and $q>p-1$. This corresponds to example (i) from the Introduction.}
\end{exm}

 \begin{cor}
 \label{c2}
Assume $g(s) \to 0$ as $s \to \infty$, as well as $(H_f)'$ and $(H_{\lambda_1})$. Assume also that there exists $s_0>0$ such that $g(s)>0$ and $g'(s)$ exists for $s \geq s_0$, with moreover   
\begin{equation}
\label{gc2}
\frac{g'(s)}{g(s)^2} \to 0 \quad \text{as} \quad s \to \infty.
\end{equation}
Then $(P)$ has at least one solution if, for some $r \in (p,p^*)$, \begin{equation}
\lim_{s \to \infty} \frac{f(x,s)g(s)^{r-1}}{e^{\frac{r-p}{p-1} G(s)}} =0 \quad \mbox{and} \quad \lim_{s \to \infty} \frac{f'(x,s)}{f(x,s)g(s)} >0
\end{equation}
uniformly with respect to $x \in \Omega$.
\end{cor} 

\begin{exm} \label{e2}
{\rm Corollary \ref{c2} applies, for instance, to $g(s)=C(1+s)^{-\alpha}$ and $f(x,s)=\lambda s^{p-1} e^{\beta s^{1-\alpha}}$,
with $C>0$, $0<\alpha<1$, $0<\beta<\frac{p^*-p}{p-1} \frac{1}{1-\alpha}$ and $\lambda<\lambda_1$. This corresponds to example (ii) of the Introduction.}
\end{exm}

Note that \eqref{gc2} implies $sg(s) \to \infty$ as $s \to \infty$. We are thus dealing in Corollary \ref{c2} with a situation where $g(s) \to 0$ but $sg(s) \to \infty$ as $s \to \infty$. In Corollary \ref{c3} below, we will consider a situation where $g(s) \to 0$ but $sg(s) \to c$ with $0\leq c <\infty$ as $s \to \infty$.

\begin{cor}
\label{c3}
Assume  $g(s) \to 0$ and $sg(s) \to c$ as $s \to \infty$, with $0\leq c<\infty$. Assume also  $(H_f)'$ and $(H_{\lambda_1})$. Then $(P)$ has at least one solution if, for some $p<r<p^*$, \begin{equation} \label{fc3}
\lim_{s \to \infty} \frac{f(x,s)}{s^{r-1}} =0 \quad \mbox{and} \quad \lim_{s \to \infty} \frac{sf'(x,s)}{f(x,s)} >p-1+c
\end{equation}
uniformly with respect to $x \in \Omega$.
\end{cor} 

\begin{exm} \label{e3}
{\rm Corollary \ref{c3} applies, for instance, to $g(s)=C(1+s)^{-\alpha}$ and $f(x,s)=s^{r-1}$, where $C>0$, $\alpha \geq 1$,  $p<r<p^*$ with $C<r-p$ if $\alpha=1$. This corresponds to example (iii) of the Introduction.}
\end{exm} 

\begin{cor}
\label{c4}
Assume $g(s) \to \infty$ as $s \to \infty$ as well as $(H_f)'$ and $(H_{\lambda_1})$. Assume also the existence of $s_0 \geq 0$ and $C<\infty$ such that $g(s)>0$, $g'(s)$ exists and $\left| \frac{g'(s)}{g(s)}\right| \leq C$ for $s\geq s_0$. Then $(P)$ has at least one solution if, for some $p<r<p^*$, 
\begin{equation}\label{fc4}
\lim_{s \to \infty} \frac{f(x,s)}{e^{\frac{r-p}{p-1} G(s)}} =0 \quad \mbox{and} \quad \lim_{s \to \infty} \frac{f'(x,s)}{f(x,s)g(s)} >0
\end{equation}
uniformly with respect to $x \in \Omega$.
\end{cor} 

\begin{exm} \label{e4}
{\rm Corollary \ref{c4} applies, for instance, to $g(s)=s^q$ and $f(x,s)=\lambda s^{p-1} e^{\beta s^{q+1}}$,
where $q>0$, $0<\beta<\frac{p^*-p}{(p-1)(q+1)}$, and $0<\lambda<\lambda_1 e^{-\beta}$. This corresponds to example (iv) of the Introduction.}
\end{exm}

Our second result in the case where $f(x,s)$ is $p$-superlinear at zero involves the following two assumptions:

 \begin{itemize}
\item [$(H_m)$] There exist $s_0 \geq 0$ such that for {\it a.e.} $x \in \Omega$ the function
$$ s \mapsto \frac{e^{G(s)}f(x,s)}{\left(\int_0^s e^{\frac{G(t)}{p-1}} dt\right)^{p-1}}$$
is nondecreasing on $[s_0,\infty)$.\\

\item  [$(H_{\infty})$] $\displaystyle \lim_{s \to \infty} \frac{e^{G(s)}f(x,s)}{\left(\int_0^s e^{\frac{G(t)}{p-1}} dt\right)^{p-1}}=\infty$ uniformly with respect to $x \in \Omega$.
\end{itemize}
 
As we will see from formulas \eqref{defa} and \eqref{defh}, the quotient $\frac{e^{G(s)}f(x,s)}{\left(\int_0^s e^{\frac{G(t)}{p-1}} dt\right)^{p-1}}$ is, up to a change of variable, equal to $\frac{h(x,s)}{s^{p-1}}$, so that $(H_m)$ corresponds to a monotonicity condition for $\frac{h(x,s)}{s^{p-1}}$, while $(H_{\infty})$ corresponds to a $p$-superlinearity condition at infinity for $h(x,s)$.

\begin{thm}
\label{t2}{\rm (Existence with a monotonicity condition)}
Assume  $(H_{SC})$, $(H_{\lambda_1})$, $(H_m)$ and $(H_{\infty})$. Then problem $(P)$ has at least one solution.
\end{thm}

Theorem \ref{t2} as well as its corollary below are proved in Section 5.

\begin{cor}
\label{ct2}
Assume $(H_f)'$, $(H_{\lambda_1})$, $(H_{SC})$, $(H_{\infty})$ and
 \begin{itemize}
\item [$(H_m)'$] $\displaystyle \lim_{s \to \infty} \frac{f'(x,s) \int_0^s e^{\frac{G(t)}{p-1}} dt}{f(x,s)\left( (p-1)e^{\frac{G(s)}{p-1}}-g(s)\int_0^s e^{\frac{G(t)}{p-1}} dt\right)} >1$
uniformly with respect to $x \in \Omega$. 
\end{itemize}
Then problem $(P)$ has at least one solution.
\end{cor}

\begin{exm} \label{e5}
{\rm Let $g(s)=\frac{p-1}{1+s}$. Then Corollary \ref{ct2} applies, for instance, to  $f(x,s)=s^{p-1}(\log(s+1))^q$ with $q>0$. This corresponds to example (v) from the Introduction.}
\end{exm}

\begin{exm}\label{e6}
{\rm Let $g(s) \equiv C$, with $C>0$. Then $(H_m)'$ reduces to $$\lim_{s \to \infty} \frac{f'(x,s)}{f(x,s)} (e^{\frac{Cs}{p-1}}-1)>C.$$
Thus Corollary \ref{ct2} applies, for instance, to $f(x,s)=s^{r-1}$ or $f(x,s)=(\log(s+1))^{r-1}$ with $p<r$ (no restriction from above is needed on $r$ in view of Proposition \ref{peq}). This corresponds respectively to examples (vi) and (vii) from the Introduction.  Corollary \ref{ct2} also applies to $f(x,s)=\lambda s^{p-1}$ or $f(x,s)=\lambda (\log(s+1))^{p-1}$ with $0<\lambda<\lambda_1$.}
\end{exm}

\begin{exm} \label{e7}
{\rm Let $g(s)=(p-1)\frac{s+2}{s+1}$. Then $(H_m)'$ reduces to $$\lim_{s \to \infty} \frac{f'(x,s)}{f(x,s)} s(s+1)>p-1.$$
Thus Corollary \ref{ct2} applies, for instance, to $f(x,s)=s^{r-1}$ or $f(x,s)=(\log(s+1))^{r-1}$ with $p<r$ (no restriction from above is needed on $r$ in view of Proposition \ref{peq}). It also applies to $f(x,s)=\lambda s^{p-1}$ or $f(x,s)=\lambda (\log(s+1))^{p-1}$ with $0<\lambda<\lambda_1$.}
\end{exm}

\begin{exm} \label{e8}
{\rm Let $g(s)=s^q$ with $q>0$. Then $(p-1)e^{\frac{G(s)}{p-1}}-g(s)\int_0^s e^{\frac{G(t)}{p-1}} dt$ is decreasing on $[0,\infty)$ and consequently $(H_m)'$ is implied by the condition 
$$ \lim_{s \to \infty} \frac{f'(x,s) \int_0^s e^{\frac{G(t)}{p-1}} dt}{f(x,s)} >1.$$
Thus Corollary \ref{ct2} applies, for instance, to $f(x,s)=s^{r-1}$ or $f(x,s)=(\log(s+1))^{r-1}$ with $p<r$ (no restriction from above is needed on $r$ in view of Proposition \ref{peq4}).}
\end{exm}

\medskip
\subsection{\textbf{The case $f(x,s)$ $p$-sublinear at zero}}
\strut
\medskip

We consider now the parametrized problem $(P_\lambda)$ and still assume $(H_g)$ and $(H_f)$.
Our first result involves the following assumptions on $f$:\\

\begin{itemize}
\item[$(H_1)$] There exists a non-empty smooth domain $\Omega_1 \subset \Omega$ such that
$$\lim_{s \to 0} \frac{f(x,s)}{s^{p-1}}=+\infty$$
uniformly with respect to $x \in \Omega_1$.\\

\item[$(H_2)$] There exist a non-empty smooth domain $\Omega_2 \subset \Omega$ and $n \in L^{\infty}(\Omega_2)$ with $n \geq 0$, $n \not \equiv 0$ such that
$$f(x,s)\geq n(x)s^{p-1}$$
for {\it a.e.} $x \in \Omega_2$ and all $s \geq 0$.\\
\end{itemize}

$(H_1)$ is a (local) $p$-sublinearity condition at zero and $(H_2)$ is related to the trivial sufficient condition of non-existence for $-\Delta u =l(u)$ in $\Omega$, $u>0$ in $\Omega$, and $u=0$ on $\partial \Omega$, namely, $\inf \{ \frac{l(s)}{s}: s>0\} >\lambda_1(\Omega)$, where $\lambda_1(\Omega)$ denotes the first eigenvalue of $-\Delta$ on $H_0^1(\Omega)$.

\begin{thm}[Existence of one solution]
\label{t3}
\strut
\begin{enumerate}
\item If $(H_1)$ holds then there exists $0<\Lambda \leq \infty$ such that $(P_\lambda)$ has at least one solution for $0<\lambda<\Lambda$ and no solution for $\lambda>\Lambda$.
\item If $(H_1)$ and $(H_2)$ hold then $\Lambda<\infty$.
\item If $(H_1)$, $(H_2)$, $(H_{SC})$ and $(H_{AR})_1$ hold then $(P_\lambda)$ has at least one solution for $\lambda=\Lambda$. 
\end{enumerate}
\end{thm}

Our purpose now is to derive a multiplicity result for $(P_\lambda)$ when $0<\lambda<\Lambda$. More assumptions on $f$ will be needed:\\

\begin{itemize}
\item[$(H_3)$] For any $s_0>0$ there exists $B=B_{f,s_0} \geq 0$ such that
for {\it a.e.} $x \in \Omega$ the function
$$s \mapsto f(x,s)+Bs^{p-1}$$
is nondecreasing on $[0,s_0]$.\\

\item[$(H_4)$] For any $u \in C_0^1(\overline{\Omega})$ with $u>0$ in $\Omega$, the function $f(x,u(x))$ is positive in $\Omega$, in the sense that for any compact $K \subset \Omega$ there exists $\epsilon>0$ such that $f(x,u(x))\geq \varepsilon$ for {\it a.e.} $x \in K$.
\end{itemize}

These two assumptions are related to the use of the strong comparison principle for the $p$-Laplacian (cf. Proposition 3.4 in \cite{DGU3}). $(H_3)$ is clearly satisfied if $f(x,s)$ is nondecreasing with respect to $s$. $(H_4)$ is satisfied, for instance, if $f$ is continuous  and positive in $\Omega \times [0,\infty)$.

\begin{thm}[Existence of two solutions]
\label{t4}
Assume $(H_1)$, $(H_{SC})$, $(H_{AR})_1$, $(H_{AR})_2$, $(H_3)$, and $(H_4)$. Then problem $(P_\lambda)$ has at least two solutions $u,v$ for $0<\lambda<\Lambda$, with $u \leq v$, $u \not \equiv v$.
\end{thm}

The examples (viii) and (ix) from the introduction illustrate Theorem \ref{t4}. Theorems \ref{t3} and \ref{t4} are proved in Section 5.

\medskip
\section{{\bf Preliminaries}}
\medskip

We first discuss the change of variable which will transform problem $(P)$ into problem $(Q)$.
 
 It was observed in \cite{KK} that, in the case $p=2$ and $g \equiv 1$, the change of variables $v=e^u-1$ transforms the quasilinear problem $(P)$ into the semilinear one
$$
\left\{
\begin{array}{ll}
-\Delta v = (1+v) f(x,\log(1+v)) \
& \mbox{in} \ \ \Omega, \ \ \\
v>0 \
&\mbox{in} \ \ \Omega, \ \ \\
 v = 0 \ &\mbox{on} \ \
\partial\Omega.
\end{array}
 \right.
$$
This can be extended to the general case of $(P)$ in the following way. Consider any change of variable
$$v=A(u),$$
where $A:[0,\infty) \rightarrow [0,\infty)$ is a $\mathcal{C}^2$ diffeomorphism with $A(0)=0$, $A(\infty)=\infty$ and $A'>0$. Clearly $u \in C^1(\overline{\Omega})$ with $u=0$ on $\partial \Omega$ if and only if $v \in C^1(\overline{\Omega})$ with $v=0$ on $\partial \Omega$. A simple computation yields
$$\Delta_p v=(p-1)A'(u)^{p-2} A''(u) |\nabla u|^p + A'(u)^{p-1} \Delta_p u$$
in the distributional sense.
It follows that $u$ solves $(P)$ if and only if $v$ satisfies
$$\Delta_p v=\left[(p-1)A'(u)^{p-2} A''(u) -g(u)A'(u)^{p-1}\right] |\nabla u|^p - A'(u)^{p-1} f(x,u).$$
The gradient term will disappear in the above expression if $A$ satisfies 
$$(p-1)A''(u)=g(u)A'(u),$$
which will be the case if one takes
\begin{equation}
\label{defa}
A(s):=\int_0^s e^{\frac{G(t)}{p-1}} \ dt.
\end{equation}
With this choice for $A$, problem $(P)$ for $u$ is equivalent to problem $(Q)$ for $v$:
$$
\left\{
\begin{array}{ll}
-\Delta_p v = h(x,v) \
& \mbox{in} \ \ \Omega, \ \ \\
v>0 \
&\mbox{in} \ \ \Omega, \ \ \\
 v = 0 \ &\mbox{on} \ \
\partial\Omega,
\end{array}
 \right.\leqno{(Q)}
$$
where 
\begin{equation}
\label{defh}
h(x,s):=e^{G(A^{-1}(s))} f(x,A^{-1}(s))
\end{equation}
Note that formulae \eqref{defa} and \eqref{defh} were already derived in \cite{HBV,MoMo}. 

 We now investigate how our assumptions on $f$ and $g$ transform into assumptions on $h$.

 Hypothesis $(H_g)$ and $(H_f)$ clearly imply in particular that $h:\Omega \times [0,\infty) \rightarrow [0,\infty)$ is a Carath\'eodory function with $h(x,s)$ remaining bounded when $s$ remains bounded. 

We recall that in the context of a problem like $(Q)$, the function $h$ is said to have subcritical growth if\\
\begin{itemize}
\item [$(SC)$] There exists $r<p^*$ such that $$\lim_{s \to \infty} \frac{h(x,s)}{s^{r-1}}=0$$ uniformly with respect to  $x \in \Omega$.\\
\end{itemize}
It is said to satisfy the Ambrosetti-Rabinowitz condition if, denoting $H(x,s):=\int_0^s h(x,t) dt$, \\
\begin{itemize}
\item [$(AR)_1$] There exist $\theta >p$ and $s_0 \geq 0$ such that $$\theta H(x,s) \leq sh(x,s)$$ for  {\it a.e.} $x \in \Omega$ and all $s \geq s_0$.\\
\item  [$(AR)_2$] There exist a non-empty smooth subdomain $\Omega_1 \subset \Omega$, $\delta>0$ and $s_0 \geq 0$ such that
$$H(x,s)\geq \delta$$ for  {\it a.e.} $x \in \Omega$ and all $s \geq s_0$.

\end{itemize}

\begin{rem}{\rm
The most usual version of the Ambrosetti-Rabinowitz condition \cite{Am-Ra} in the present $p$-Laplacian context deals with a continuous function $h(x,s)$ on $\overline{\Omega} \times [0,\infty)$ and requires the existence of $\theta>p$ and $s_0 \geq 0$ such that
\begin{equation}
\label{aru}
0<\theta H(x,s) \leq sh(x,s)
\end{equation}
for all $x \in \overline{\Omega}$ and all $s \geq s_0$. Condition \eqref{aru} clearly implies $(AR)_1$ and $(AR)_2$. Some care must however be taken when dealing with a continuous (and a fortiori a Carath\'eodory) function on $\Omega \times [0,\infty)$, as was observed recently in \cite{Mu}. This is the reason for the present formulation of $(AR)_2$. Note also that $(AR)_2$ can be seen as a localized version of the first inequality in \eqref{aru}.}
\end{rem}

\begin{lem}\label{l1}
The function $h$ from \eqref{defh} satisfies $(SC)$ if and only if the functions $f$ and $g$ satisfy $(H_{SC})$.
\end{lem}

\begin{proof}
Writing $A^{-1}(s)=t$ in \eqref{defh}, replacing  in $(SC)$, and using \eqref{defa}, the equivalence follows immediately.
\end{proof}

\begin{lem}\label{l2}
The function $h$ from \eqref{defh} satisfies $(AR)_1$ and $(AR)_2$ if and only if the functions $f$ and $g$ satisfy $(H_{AR})_1$ and $(H_{AR})_2$.
\end{lem}

\begin{proof}
Writing $A^{-1}(s)=t$ in \eqref{defh}, replacing  in $(AR)_1$ and $(AR)_2$, and using \eqref{defa}, the equivalence follows immediately.
\end{proof}

\begin{lem}\label{l3}
The function $h$ from \eqref{defh} satisfies 
$$\limsup_{s \to 0} \frac{h(x,s)}{s^{p-1}} =\limsup_{t \to 0} \frac{f(x,t)}{t^{p-1}}.$$
The same conclusion holds with $\limsup$ replaced on both sides by either $\liminf$ or $\lim$.
\end{lem}

\begin{proof}
By L'Hospital rule, we have
\begin{eqnarray*}
\lim_{t \to 0} \frac{t A'(t)}{A(t)}&=&\lim_{t \to 0} \frac{t e^{\frac{G(t)}{p-1}}}{\int_0^t e^{\frac{G(\tau)}{p-1}} \ d\tau}\\&=&\lim_{t \to 0} \frac{e^{\frac{G(t)}{p-1}}+(p-1)^{-1}t g(t)e^{\frac{G(t)}{p-1}}}{e^{\frac{G(t)}{p-1}}}\\&=&\lim_{t \to 0} (1+ (p-1)^{-1} tg(t))\\&=&1.
\end{eqnarray*}
Writing $A^{-1}(s)=t$ in \eqref{defh}, it follows that 
\begin{eqnarray*}
\limsup_{s \to 0} \frac{h(x,s)}{s^{p-1}}&=& \limsup_{t \to 0} \left(\frac{A'(t)}{A(t)}\right)^{p-1} f(x,t)\\&=&\limsup_{t \to 0}  \left(\frac{tA'(t)}{A(t)}\right)^{p-1} \frac{f(x,t)}{t^{p-1}}\\&=&  \limsup_{t \to 0}   \frac{f(x,t )}{t^{p-1}}.
\end{eqnarray*}
The same argument applies with $\liminf$ or $\lim$.
\end{proof}

\medskip
 \section{\textbf{Subcritical growth and the Ambrosetti-Rabinowitz condition} }
\medskip 

In this section we provide sufficient conditions to have $(H_{SC})$, $(H_{AR})_1$, and $(H_{AR})_2$. First we show that under the further assumption $(H_f)'$, the conditions $(H_{AR})_1$ and $(H_{AR})_2$ are implied by the following condition $(H_{AR})'$, which is much easier to verify in most cases.

\begin{prop}\label{thlim}
Assume $(H_f)'$. If 
$$
\lim_{s \to \infty} \frac{\int_0^s e^{\frac{G(t)}{p-1}} dt}{e^{\frac{G(s)}{p-1}}} \left( g(s) + \frac{f'(x,s)}{f(x,s)}\right) >p-1  \leqno{(H_{AR})'}
$$
uniformly with respect to $x \in \Omega$, then $(H_{AR})_1$ and $(H_{AR})_2$ hold with $\Omega_1=\Omega$.
\end{prop}

\begin{proof}
Note first that by $(H_f)'$ we have, for $s$ large and for some $\varepsilon>0$, $f(x,s) \geq \varepsilon$, which clearly implies $(H_{AR})_2$ with $\Omega_1=\Omega$. We prove now $(H_{AR})_1$. To this end, it is enough to show that
\begin{equation}
\label{hlim}
\lim_{s \to \infty} \frac{f(x,s)e^{G(s)}\int_0^s e^{\frac{G(t)}{p-1}} dt}{\int_0^s e^{\frac{pG(t)}{p-1}} f(x,t) dt} >p
\end{equation}
uniformly with respect to $x \in \Omega$. Using L'Hospital rule in \eqref{hlim}, one easily concludes that $(H_{AR})'$ implies that the limit in \eqref{hlim} exists and is larger than $p$.
\end{proof}

\begin{rem}\label{rexp}
{\rm Let us assume $g \equiv 1$. We claim that a necessary condition for $(H_{AR})_1$ and $(H_{AR})_2$ to hold is that $F(x,s)$ has at least an exponential growth on $\Omega_1$, i.e. there exist $C_1,C_2>0$ and $s_0 \geq 0$ such that
\begin{equation}
\label{fexp}
F(x,s) \geq C_1e^{C_2 s}
\end{equation}
for all $s \geq s_0$ and {\it a.e.} $x \in \Omega_1$. Indeed, $(H_{AR})_1$ and $(H_{AR})_2$ with $g \equiv 1$ imply the existence of $\mu$ such that
$$\frac{p}{p-1} <\mu \leq \liminf_{s \to \infty} \frac{e^{\frac{ps}{p-1}}f(x,s)}{\int_0^s e^{\frac{pt}{p-1}}f(x,t) dt}$$
for {\it a.e.} $x \in \Omega_1$. One deduces from the above inequality the existence of $C>0$ and $s_1 \geq 0$ such that
$$\int_0^s e^{\frac{pt}{p-1}}f(x,t) dt \geq C e^{\mu s}$$
for all $s \geq s_1$ and {\it a.e.} $x \in \Omega_1$. Integration by parts gives $$F(x,s) \geq C_1e^{\left(\mu - \frac{p}{p-1}\right) s}$$
for all $s \geq s_1$ and {\it a.e.} $x \in \Omega_1$, which provides \eqref{fexp}. It follows in particular from \eqref{fexp} that, when $g \equiv 1$, any $f(x,s)$ of polynomial type with respect to $s$ does not satisfy $(H_{AR})_1$ and $(H_{AR})_2$.}
\end{rem}

The four propositions below describe the form taken by conditions $(H_{SC})$ and $(H_{AR})'$ in various cases, which are classified according to the behavior of $g$ and $f$ at infinity. Proposition \ref{peq} (respect. \ref{peq2}, \ref{peq3}, \ref{peq4}) together with Theorem \ref{t1} and Proposition \ref{thlim} clearly implies Corollary \ref{c1} (respect. \ref{c2}, \ref{c3}, \ref{c4}).

\begin{prop}\label{peq}
Assume $(H_f)'$ and $g(s) \to g_{\infty}$ as $s \to \infty$, with $0<g_\infty<\infty$. Then:
\begin{enumerate}
\item $(H_{SC})$ holds if and only if for some $r$ with $r<p^*$, there holds
\begin{equation}
\label{f1}
\lim_{s \to \infty} \frac{f(x,s)}{e^{\frac{r-p}{p-1} G(s)}} =0
\end{equation}
uniformly with respect to $x \in \Omega$.\\
\item $(H_{AR})'$ holds if and only if 
\begin{equation}
\label{f2}
\lim_{s \to \infty} \frac{f'(x,s)}{f(x,s)} >0
\end{equation}
uniformly with respect to $x \in \Omega$.
\end{enumerate}
\end{prop}

Note that \eqref{f2} implies that $f$ must have exponential growth.

\begin{proof}
\strut
\begin{enumerate}
\item Since  $0<g_{\infty}<\infty$ we have $\dis \lim_{s \to \infty} G(s)=\infty$, so that $$\dis  \lim_{s \to \infty} \int_0^s e^{\frac{G(t)}{p-1}} \ dt=\lim_{s \to \infty} e^{\frac{G(s)}{r-1}}=\infty.$$ We use L'Hospital rule to obtain
$$\lim_{s \to \infty} \frac{e^{\frac{G(s)}{r-1}}}{ \int_0^s e^{\frac{G(t)}{p-1}} \ dt}=\lim_{s \to \infty} \frac{g(s)}{r-1} e^{\frac{(p-r)G(s)}{(p-1)(r-1)}}.$$
Since 
$$
\lim_{s \to \infty} \frac{e^{G(s)}f(x,s)}{\left( \int_0^s e^{\frac{G(t)}{p-1}} d t\right)^{r-1}}= \left(\lim_{s \to \infty} \frac{e^{\frac{G(s)}{r-1}}f(x,s)^{\frac{1}{r-1}}}{\int_0^s e^{\frac{G(t)}{p-1}} dt}\right)^{r-1},
$$
we see that $(H_{SC})$ holds if and only if $$\lim_{s \to \infty} \frac{e^{\frac{G(s)}{r-1}}f(x,s)^{\frac{1}{r-1}}}{\int_0^s e^{\frac{G(t)}{p-1}} dt}=0.$$
From $$\lim_{s \to \infty} \frac{e^{\frac{G(s)}{r-1}}}{ \int_0^s e^{\frac{G(t)}{p-1}} \ dt} \frac{p-1}{g(s) e^{\frac{(p-r)G(s)}{(p-1)(r-1)}} }=\lim_{s \to \infty} 
\frac{p-1}{g(s)} \frac{e^{\frac{G(s)}{p-1}}}{ \int_0^s e^{\frac{G(t)}{p-1}} \ dt} =1,$$
we infer 
\begin{eqnarray*}
\lim_{s \to \infty} \frac{e^{\frac{G(s)}{r-1}}f(x,s)^{\frac{1}{r-1}}}{\int_0^s e^{\frac{G(t)}{p-1}} dt}&=&\lim_{s \to \infty} \frac{g(s)}{p-1} f(x,s)^{\frac{1}{r-1}}e^{\frac{(p-r)G(s)}{(p-1)(r-1)}}\\&=& \frac{g_\infty}{p-1}\lim_{s \to \infty} \left( \frac{f(x,s)}{e^{\frac{r-p}{p-1} G(s)}}\right)^{\frac{1}{r-1}}
\end{eqnarray*}
%Now, since $g_{\infty} \in (0,\infty)$ it follows that $\lim_{t \to \infty} \frac{G(t)}{g_\infty t}=1$, so 
%$$\lim_{t \to \infty} \frac{g_\infty}{p-1}\lim_{t \to \infty} \left( \frac{f(x,t)}{e^{\frac{r-p}{p-1} G(t)}}\right)^{\frac{1}{r-1}}= \lim_{t \to \infty} \frac{g_\infty}{p-1}\lim_{t \to \infty} \left( \frac{f(x,t)}{e^{\frac{r-p}{p-1} g_\infty t}}\right)^{\frac{1}{r-1}}.$$
Thus, we conclude that $(H_{SC})$ holds if and only if, for some $r<p^*$, there holds
$$\lim_{s \to \infty} \frac{f(x,s)}{e^{\frac{r-p}{p-1} G(s)}}=0$$
uniformly with respect to $x \in \Omega$.\\

\item By L'Hospital rule, we have
$$\lim_{s \to \infty} \frac{ \int_0^s e^{\frac{G(t)}{p-1}} \ dt} {e^{\frac{G(s)}{p-1}}}=\frac{p-1}{g_{\infty}}.$$
Thus $$ \lim_{s \to \infty} \left(\int_0^s e^{\frac{G(t)}{p-1}} dt\right) e^{-\frac{G(s)}{p-1}}\left[ g(s) + \frac{f'(x,s)}{f(x,s)}\right]>p-1$$ if and only if
$$ \lim_{s \to \infty} g(s)^{-1}\left( g(s) + \frac{f'(x,s)}{f(x,s)}\right) >1,$$
i.e. $$\lim_{s \to \infty} \frac{f'(x,s)}{f(x,s)} >0.$$
\end{enumerate}
\end{proof}

\begin{prop}\label{peq2}
Assume $(H_f)'$ and $g(s) \to 0$ as $s \to \infty$. In addition, assume that there exists $s_0>0$ such that $g'(s)$ exists for $s>s_0$ and  $\frac{g'(s)}{g(s)^2} \to 0$ as $s \to \infty$. Then:
\begin{enumerate}
\item $(H_{SC})$ holds if and only if for some $r$ with $r<p^*$, there holds
\begin{equation}
\label{f3}
\lim_{s \to \infty} \frac{f(x,s)g(s)^{r-1}}{e^{\frac{r-p}{p-1} G(s)}} =0
\end{equation}
uniformly with respect to $x \in \Omega$.\\
\item $(H_{AR})'$ holds if and only if 
\begin{equation}
\label{f4}
\lim_{s \to \infty} \frac{f'(x,s)}{f(x,s)g(s)} >0
\end{equation}
uniformly with respect to $x \in \Omega$.
\end{enumerate}
\end{prop}

\begin{proof}
\strut
\begin{enumerate}
\item First note that by L'Hospital rule  $$\dis \lim_{s \to \infty} sg(s)=\lim_{s \to \infty} \frac{s}{g(s)^{-1}}=\lim_{s \to \infty} - \left( \frac{g'(s)}{g(s)^2}\right)^{-1}=\infty,$$ and consequently $\dis \lim_{s \to \infty} G(s)=\infty$. Hence
$$
 \lim_{s \to \infty} \frac{e^{\frac{G(s)}{r-1}}}{ \int_0^s e^{\frac{G(t)}{p-1}} \ dt} \frac{p-1}{g(s) e^{\frac{(p-r)G(s)}{(p-1)(r-1)}} }=\lim_{s \to \infty} 
(p-1)\frac{g(s)^{-1}e^{\frac{G(s)}{p-1}}}{ \int_0^s e^{\frac{G(t)}{p-1}} \ dt} .$$
By L'Hospital rule, the latter limit is equal to
$$(p-1)\lim_{s \to \infty} \left(\frac{1}{p-1}-\frac{g'(s)}{g(s)^2}\right)=1.$$
Thus we have
$$
 \lim_{s \to \infty} \left( \frac{ e^{\frac{G(s)}{r-1}}}{\int_0^s e^{\frac{G(t)}{p-1}} dt}\right)^{r-1} f(x,s)=  \lim_{s \to \infty} \left( \frac{g(s)}{p-1}\right)^{r-1} \frac{f(x,s)}{e^{\frac{r-p}{p-1} G(s)}}=0
$$
if and only if $$\lim_{s \to \infty} \frac{g(s)^{r-1}f(x,s)}{e^{\frac{r-p}{p-1} G(s)}}=0$$
for some $r<p^*$.\\

\item Note that
$$\lim_{s \to \infty} g(s) \frac{ \int_0^s e^{\frac{G(t)}{p-1}} \ dt} {e^{\frac{G(s)}{p-1}}}=\lim_{s \to \infty}  \frac{ \int_0^s e^{\frac{G(t)}{p-1}} \ dt} {e^{\frac{G(s)}{p-1}}g(s)^{-1}}.$$
By L'Hospital rule, the latter limit is equal to
$$\lim_{s \to \infty} \left( \frac{1}{p-1} -\frac{g'(s)}{g(s)^2}\right)^{-1}=p-1.$$
Finally, since $$\lim_{s \to \infty} \frac{ \int_0^s e^{\frac{G(t)}{p-1}} \ dt} {e^{\frac{G(s)}{p-1}}}=\lim_{s \to \infty}  \frac{p-1}{g(s)}$$
we infer that $(H_{AR})'$ holds if and only if
$$
\lim_{s \to \infty} \frac{f'(x,s)}{f(x,s)g(s)} >0
$$
uniformly with respect to $x \in \Omega$.
\end{enumerate}
\end{proof}

\begin{prop}\label{peq3}
Assume $(H_f)'$ and  $g(s) \to 0$ with $sg(s) \to c$ as $s \to \infty$, with $0\leq c<\infty$. Then:
\begin{enumerate}
\item $(H_{SC})$ holds if and only if for some $r<p^*$, there holds
\begin{equation}
\lim_{s \to \infty} \frac{f(x,s)}{s^{r-1}} =0
\end{equation}
uniformly with respect to $x \in \Omega$.\\
\item $(H_{AR})$ holds if and only if 
\begin{equation}
\lim_{s \to \infty} \frac{sf'(x,s)}{f(x,s)} >p-1
\end{equation}
uniformly with respect to $x \in \Omega$.
\end{enumerate}
\end{prop}

\begin{proof}
\strut
\begin{enumerate}
\item Note that $$
\lim_{s \to \infty} \frac{s e^{\frac{G(s)}{r-1}}}{\int_0^s e^{\frac{G(t)}{p-1}} dt} =\lim_{s \to \infty} \frac{1+\frac{s g(s)}{r-1}}{e^{\frac{(r-p)G(s)}{(p-1)(r-1)}}} <\infty,
$$
since $e^{\frac{(r-p)G(s)}{(p-1)(r-1)} } \geq 1$ for every $s>0$. Thus
$$
 \lim_{s \to \infty} \left( \frac{ e^{\frac{G(s)}{r-1}}}{\int_0^s e^{\frac{G(t)}{p-1}} dt}\right)^{r-1} f(x,s)= \lim_{s \to \infty} \left( \frac{s e^{\frac{G(s)}{r-1}}}{\int_0^s e^{\frac{G(t)}{p-1}} dt}\right)^{r-1} \frac{f(x,s)}{s^{r-1}}=0$$
if and only if $$\lim_{s \to \infty} \frac{f(x,s)}{s^{r-1}}=0$$ for some  $r<p^*$.\\
\item Since $g(s)\to 0$ as $s \to \infty$ and
\begin{eqnarray*}
\lim_{s \to \infty} \frac{\int_0^s e^{\frac{G(t)}{p-1}}}{se^{\frac{G(s)}{p-1}}}&=&\lim_{s \to \infty} \left(1+ \frac{sg(s)}{p-1}\right)^{-1}\\&=&\frac{p-1}{p-1+c},
\end{eqnarray*}
we have 
$$
\lim_{s \to \infty} \frac{\int_0^s e^{\frac{G(t)}{p-1}} \ dt}{se^{\frac{G(s)}{p-1}}} \left(sg(s)+\frac{sf'(x,s)}{f(x,s)}\right)= \frac{p-1}{p-1+c} \left(c+\lim_{s \to \infty} \frac{sf'(x,s)}{f(x,s)}\right).
$$
Thus $(H_{AR})'$ holds if and only if 
$$\liminf_{s \to \infty}  \frac{sf'(x,s)}{f(x,s)}>p-1$$
uniformly with respect to $x \in \Omega$.
\end{enumerate}
\end{proof}

\begin{prop}\label{peq4}
Assume $(H_f)'$ and $g(s) \to \infty$ as $s \to \infty$, with moreover for some $C,s_0 \geq 0$, $\left| \frac{g'(s)}{g(s)}\right| \leq C$ for $s\geq s_0$. Then:
\begin{enumerate}
\item $(H_{SC})$ holds if and only if for some  $r<p^*$, there holds
\begin{equation}
\lim_{s \to \infty} \frac{f(x,s)}{e^{\frac{r-p}{p-1} G(s)}} =0
\end{equation}
uniformly with respect to $x \in \Omega$.\\
\item $(H_{AR})$ holds if and only if 
\begin{equation}
\lim_{s \to \infty} \frac{f'(x,s)}{f(x,s)g(s)} >0
\end{equation}
uniformly with respect to $x \in \Omega$.
\end{enumerate}
\end{prop}

\begin{proof}
\strut
\begin{enumerate}
\item By L'Hospital rule, we have
$$\lim_{s \to \infty} g(s) e^{\frac{(p-r)G(s)}{(p-1)(r-1)}}=\lim_{s \to \infty} \frac{g'(s)}{g(s)}\frac{(p-1)(r-1)}{r-p} e^{\frac{(p-r)G(s)}{(p-1)(r-1)}}.$$
Moreover, $$\lim_{s \to \infty} \frac{e^{\frac{G(s)}{r-1}}}{ \int_0^s e^{\frac{G(t)}{p-1}} \ dt}=\lim_{s \to \infty} \frac{g(s)}{r-1} e^{\frac{(p-r)G(s)}{(p-1)(r-1)}}.$$
Hence, 
\begin{eqnarray*}
\lim_{s \to \infty} \frac{e^{\frac{G(s)}{r-1}}}{ \int_0^s e^{\frac{G(t)}{p-1}} \ dt} f(x,s)^{\frac{1}{r-1}}&=&\lim_{s \to \infty} \frac{g(s)}{r-1} e^{\frac{(p-r)G(s)}{(p-1)(r-1)}}f(x,s)^{\frac{1}{r-1}}\\&=&\lim_{s \to \infty} \frac{g'(s)}{g(s)}\frac{p-1}{r-p} f(x,s)^{\frac{1}{r-1}} e^{\frac{p-r}{(p-1)(r-1)}G(s)}
\end{eqnarray*}
Thus $(H_{SC})$ holds if and only if, for some $r<p^*$, there holds $$\lim_{s \to \infty} \frac{f(x,s)}{e^{\frac{r-p}{p-1} G(s)}}=0$$
uniformly with respect to $x \in \Omega$.\\

\item %First of all, since $g(t) \to \infty$ as $t \to \infty$, we have
%$\frac{A(t)}{A'(t)}\to 0$ as $t \to \infty$. 
Note that $(H_{AR})'$ holds if and only if
$$\lim_{s \to \infty} \frac{g(s)\int_0^s e^{\frac{G(t)}{p-1}}dt}{e^{\frac{G(s)}{p-1}}}\left(1+ \frac{f'(x,s)}{f(x,s)g(s)}\right)>p-1.$$
Now, by L'Hospital rule, we have
\begin{eqnarray*}
 \lim_{s \to \infty} \frac{g(s)\int_0^s e^{\frac{G(t)}{p-1}}dt}{e^{\frac{G(s)}{p-1}}}&=&  \lim_{s \to \infty} \frac{g(s)e^{\frac{G(s)}{p-1}}+g'(s)\int_0^s e^{\frac{G(t)}{p-1}}dt}{(p-1)^{-1}e^{\frac{G(s)}{p-1}}g(s)}\\&=& p-1 +(p-1)\lim_{s \to \infty} \frac{g'(s)}{g(s)} \frac{\int_0^s e^{\frac{G(t)}{p-1}}dt}{e^{\frac{G(s)}{p-1}}}\\&=&p-1
\end{eqnarray*}
Therefore $(H_{AR})'$ holds if and only if
$$\lim_{s \to \infty} \frac{f'(x,s)}{f(x,s)g(s)} >0.$$
\end{enumerate}
\end{proof}

\medskip
 \section{\textbf{Proofs of the Theorems and Corollary \ref{ct2}}}
\medskip 
 
\noindent {\bf Proof of Theorem \ref{t1}:}
We study $(P)$ in its equivalent form $(Q)$. By Lemma \ref{l1} and Lemma \ref{l2}, the function $h$ given by \eqref{defh} has subcritical growth and satisfies the Ambrosetti-Rabinowitz condition. Moreover, by $(H_{\lambda_1})$ and Lemma \ref{l3}
\begin{equation}\label{3h}
\limsup_{s \to 0} \frac{h(x,s)}{s^{p-1}} <\lambda_1
\end{equation}
uniformly with respect to $x \in \Omega$. The conclusion is now a consequence of the standard result recalled below as Proposition \ref{mp}. \qed

\begin{prop}
\label{mp}
Let $h:\Omega \times [0,\infty) \rightarrow [0,\infty)$ be a Carath\'eodory function with $h(x,s)$ remaining bounded when $s$ remains bounded. Assume that $h$ satisfies $(SC)$, $(AR)_1$, $(AR)_2$ and \eqref{3h}. Then problem $(Q)$ has at least one solution in $C^1(\overline{\Omega})$.
\end{prop}

\noindent {\bf Sketch of the proof of Proposition \ref{mp}:}
Note that \eqref{3h} implies $h(x,0)\equiv 0$. Extend $h$ to $\Omega \times \Re$ by putting $h(x,s)=0$ for $s<0$ and consider the functional
$$J(v):= \frac{1}{p} \int_\Omega |\nabla v|^p - \int_\Omega H(x,v)$$
on $W_0^{1,p}(\Omega)$, where $H(x,s):=\int_0^s h(x,t) dt$.

It is a standard matter to apply the mountain-pass theorem to $J$. Assumption \eqref{3h} leads to the existence of a mountain-pass range around zero, while assumptions $(AR)_1$ and $(AR)_2$ allow the construction of a direction in which $J$ goes to $-\infty$. Assumption $(AR)_1$ implies the boundedness of the Palais-Smale sequences, and assumption $(SC)$ together with the $(S)_+$ property of the $p$-Laplacian provide the required compactness.

One obtains in this way a nontrivial solution $v \in W_0^{1,p}(\Omega)$ of the equation in $(Q)$. One then derives from \cite{An} or \cite{Gu-Ve} that $v$ belongs to $L^{\infty}(\Omega)$. This allows the application of \cite{Lie} to obtain that $v$ belongs to some $C^{1,\alpha}(\overline{\Omega})$. Finally, using $v^-$ as test function in the equation in $(Q)$, one deduces $v \geq 0$. The strong maximum principle of \cite{Va} then yields $v>0$ in $\Omega$ with $\frac{\partial v}{\partial \nu}<0$ on $\partial \Omega$, where $\nu$ denotes the exterior normal. \qed

\medskip

\noindent {\bf Proof of Theorem \ref{t2}:}
It consists in applying Theorem 1.1 from \cite{Liu} to $h(x,s)$.
First of all, $(H_{SC})$ implies that  $h(x,s)$ is a subcritical function. Moreover, as already observed in Section 2, $(H_{\infty})$ implies that $h(x,s)$ has a $p$-superlinear behaviour at infinity, while $(H_m)$ implies that $\frac{h(x,s)}{s^{p-1}}$ is non-decreasing for $s$ large enough. Finally, we have the $p$-superlinearity of $h$ at zero, by hypothesis $(H_{\lambda_1})$ and Lemma \ref{l3}.  Therefore all the hypotheses of Theorem 1.1 from \cite{Liu} are fulfilled, and it follows that problem $(Q)$ has at least one solution. The proof of Theorem \ref{t2} is complete. \qed \\

\noindent {\bf Proof of Corollary \ref{ct2}:}
Using $(H_f)'$, we can compute the derivative (with respect to $s$) of the quotient
$$\frac{e^{G(s)}f(x,s)}{\left(\int_0^s e^{\frac{G(t)}{p-1}} dt\right)^{p-1}},$$ which is given by  $$\frac{e^{G(s)}}{\left(\int_0^s e^{\frac{G(t)}{p-1}} dt\right)^p} \left[ \left(g(s)f(x,s)+f'(x,s)\right)\int_0^s e^{\frac{G(t)}{p-1}} dt-(p-1)f(x,s)e^{\frac{G(s)}{p-1}}\right].$$
Thus $(H_m)$ holds if 
$$f'(x,s)\int_0^s e^{\frac{G(t)}{p-1}} dt \geq f(x,s) \left( (p-1)e^{\frac{G(s)}{p-1}}-g(s)\int_0^s e^{\frac{G(t)}{p-1}} dt\right)$$
for $s$ sufficiently large. This condition is clearly satisfied if $(H_m)'$ holds. Theorem \ref{t2} provides the conclusion. \qed \\

\noindent {\bf Proof of Theorem \ref{t3}:}
It consists in applying Theorems 2.1, 2.2 and 2.3 from \cite{DGU3} to the transformed problem $(Q_\lambda)$. The point is thus to verify that the right-hand side of $(Q_\lambda)$, $\lambda h(x,s)$, satisfies the hypothesis of these theorems.
The hypothesis $(H)$ and $(H_0)$ from \cite{DGU3} are trivially satisfied. Hypothesis $(H_e)$ from \cite{DGU3} can be verified by taking $g(s)=\lambda \sup \{h(x,t): x \in \Omega, t \leq s\}$ and choosing $\lambda>0$ sufficiently small. Hypothesis $(H_{\Omega_1})$ from \cite{DGU3} follows by applying Lemma \ref{l3} to $(H_1)$ above. Consequently, Theorem 2.1 from \cite{DGU3} applies and yields $(1)$.

To prove $(2)$ one writes $A^{-1}(s)=t$ in \eqref{defh} to deduce
\begin{equation}\label{ehs}
\frac{h(x,s)}{s^{p-1}}=\left(\frac{te^{\frac{G(t)}{p-1}}}{\int_0^t e^{\frac{G(\tau)}{p-1}} dt} \right)^{p-1} \frac{f(x,t)}{t^{p-1}}.
\end{equation}
Using L'Hospital rule, one sees that the bracket in \eqref{ehs} has positive limits (possibly $+\infty$) when $t \to 0$ and $t \to \infty$, and consequently, this bracket has a positive infimum $\alpha$ for $t \in (0,\infty)$. It then follows from $(H_2)$ that
$$h(x,s) \geq \alpha n(x) s^{p-1}$$
for {\it a.e.} $x \in \Omega_2$ and all $s \geq 0$. Consequently hypothesis $(H_{\tilde{\Omega}})$ from \cite{DGU3} is satisfied by $\lambda h(x,s)$ after taking $\lambda$ sufficiently large. Theorem 2.2 from \cite{DGU3} then applies and yields $(2)$. 

Finally, since $(H_{SC})$ for $f,g$ implies $(SC)$ for $h$ (cf. Lemma \ref{l1}) and $(H_{AR})_1$ for $f,g$ implies $(AR)_1$ for $h$ (cf. Lemma \ref{l2}), Theorem 2.3 from \cite{DGU3} applies and yields $(3)$.
\qed \\

\noindent {\bf Proof of Theorem \ref{t4}:}
It consists in applying Theorem 2.4  from \cite{DGU3} to the transformed problem $(Q_\lambda)$. There are three hypothesis from \cite{DGU3}  which have to be verified by $\lambda h(x,s)$: $(H_0)'$, $(M)$ and $(H_{\Omega_2})$.
First of all, $(M)$ from \cite{DGU3} is a direct consequence of $(H_4)$.

To verify $(H_{\Omega_2})$ from \cite{DGU3}, one observes that $(H_{AR})_1$ and $(H_{AR})_2$ imply that $h$ satisfies $(AR)_1$ and $(AR)_2$. A standard integration then leads to
$$H(x,s) \geq \nu s^{\theta}$$
for some constants $\nu>0$, $\theta >p$ and for {\it a.e.} $x \in \Omega_1$ and all $s$ large. This is more than needed to have $(H_{\Omega_1})$ from \cite{DGU3}.

The verification of $(H_0)'$ from \cite{DGU3} requires some computation. We have to show that for any $s_0>0$ there exists $B$ such that for {\it a.e.} $x \in \Omega$ the function
$$s \mapsto h(x,s)+Bs^{p-1}$$ 
is nondecreasing on $[0,s_0]$. Writing $A^{-1}(s)=t$ in \eqref{defh}, this amounts to showing that
\begin{equation}
\label{aeq}
t \mapsto e^{G(t)} f(x,t) +B A(t)^{p-1}
\end{equation}
is nondecreasing on $[0,A^{-1}(s_0)]$. We chose the constant $B_{f,A(s_0)}$ provided by $(H_3)$, call it $B_1$, and write the expression in \eqref{aeq} as
\begin{equation}
\label{aeqt}
e^{G(t)}\left( f(x,t)+B_1 t^{p-1}\right) + a(t)
\end{equation}
where
$$a(t)=-B_1 t^{p-1} e^{G(t)} + BA(t)^{p-1}.$$
The first term in \eqref{aeqt} is nondecreasing for $t \in [0,A^{-1}(s_0)]$ by the choice of $B_1$. We will see that by taking $B$ sufficiently large, $a(t)$ is also nondecreasing on $[0,A^{-1}(s_0)]$. One computes the derivative
$$a'(t)=A(t)^{p-2} \left[ B(p-1) e^{\frac{G(t)}{p-1}} -B_1g(t)t e^{G(t)} \frac{t^{p-2}}{A(t)^{p-2}} - B_1(p-1)e^{G(t)} \frac{t^{p-2}}{A(t)^{p-2}}\right].$$
Using the fact that $\frac{A(t)}{t} \to 1$ as $t \to 0$, one easily sees that if $B$ is taken sufficiently large then $a'(t)>0$ for $t \in [0,A^{-1}(s_0)]$. The conclusion follows.
\qed


\begin{thebibliography}{1000}

\bibitem{ADP}  {\rm B. Abdellaoui, A. Dall'Aglio, I. Peral}, Ireneo Some remarks on elliptic problems with critical growth in the gradient. {\em J. Differential Equations} 222 (2006), no. 1, 21–62.

\bibitem{ABC} {\rm A.\ Ambrosetti, H.\ Brezis, and G.\ Cerami}, Combined effects of concave and convex nonlinearities in some elliptic problems, {\em J.\ Functional Analysis} {\bf 122}, (1994), 519--543. 

\bibitem{Am-Ra}  {\rm A. Ambrosetti, P. Rabinowitz},  Dual variational methods in critical point theory and applications. {\em J.\ Functional Analysis} 14 (1973), 349-–381. 

\bibitem{An} {\rm A. Anane} Etude des valeurs propres et de la r\'esonance pour l'op\'erateur $p$-Laplacien. Th\`ese de doctorat, Universit\'e Libre de Bruxelles, 1988.

\bibitem{BoMuPu} {\rm  L. Boccardo, F. Murat, J.P. Puel}, $L^{\infty}$ Estimate for some nonlinear elliptic partial differential equations and application to an existence result. 
{\em SIAM J. Math. Anal. } Vol. 23 (1992), no 2, 326--333.

\bibitem{BrNi} {\rm H. Br\'ezis, L. Nirenberg}, Positive solutions of nonlinear elliptic equations involving critical Sobolev exponents.
{\em Comm.Pure Appl. Math. } Vol.36 (1983),  437--477.

\bibitem{BEZF} {\rm  H. Bueno, G. Ercole, A. Zumpano, W. M. Ferreira}, Positive solutions for the p-Laplacian with dependence on the gradient. {\em Nonlinearity} 25 (2012), no. 4, 1211–-1234.

\bibitem{Eg} {\rm H. Egnell}, Existence and nonexistence results for m-Laplace equations involving critical Sobolev exponents. {\em Arch. Rational Mech. Anal. } {\bf 104} (1988),  57--77.

\bibitem{FaMiMoTa} {\rm L.F.O. Faria, O.H. Miyagaki, D. Motreanu, M. Tanaka}, Existence result for nonlinear elliptic equations with Leray-Lions operator and dependence on the gradient.
{\em Nonlinear Analysis  }  (2014), no 96, 154--166.



\bibitem{GaPe} {\rm J. Garcia, I. Peral}, Existence and nonuniqueness for the p-Laplacian: nonlinear eigenvalues,
{\em Comm. Partial Differential Equations  }  (1987), no 12, 1389--1430.

\bibitem{GueVe} {\rm M Gueda, L. V\'eron}, Quasilinear Equations involving critical Sobolev exponents,
{\em Nonlinear Analysis }  (1989), no 13, 879--902.



\bibitem{DGU3} {\rm D. G. De Figueiredo,  J-P. Gossez, P. Ubilla},
 Local ``superlinearity'' and ``sublinearity'' for the $p$-Laplacian.
{\em J. Funct. Anal.} {\bf 257} (2009), no 3, 721--752.

\bibitem{GMIRQ} {\rm J. Garc\'ia-Melián, L. Iturriaga, H. Ramos Quoirin}, A priori bounds and existence of solutions for slightly superlinear elliptic problems. {\em Adv. Nonlinear Stud.} 15 (2015), no. 4, 923–-938.

\bibitem{Gu-Ve}  {\rm M. Guedda, L. V\'eron},  Quasilinear elliptic equations involving critical Sobolev exponents. {\em Nonlinear Anal.} 13 (1989), no. 8, 879–-902.

\bibitem{HBV}  {\rm H. A. Hamid, M. F. Bidaut-Veron}, On the connection between two quasilinear elliptic problems with source terms of order 0 or 1. {\em Commun. Contemp. Math.} {\bf 12} (2010), no. 5, 727–-788.

\bibitem{ILS} {\rm L. Iturriaga, S. Lorca, J. S\'anchez}  Existence and multiplicity results for the $p-$Laplacian with a $p-$Gradient term {\em NoDEA. Nonlinear Differ. Equ. Appl.} 15 (2008), 729-–743.

\bibitem{ILU} {\rm L. Iturriaga, S. Lorca, P. Ubilla},  A quasilinear problem without the Ambrosetti-Rabinowitz-type condition. {\em Proc. Roy. Soc. Edinburgh Sect. A} 140 (2010), no. 2, 391-–398.

\bibitem{JRQ}  {\rm L. Jeanjean, H. Ramos Quoirin}, Multiple solutions for an indefinite elliptic problem with critical growth in the gradient. {\em Proc. Amer. Math. Soc.} 144 (2016), no. 2, 575–-586

\bibitem{JS}
{\rm L.\ Jeanjean, B.\ Sirakov}, Existence and multiplicity for elliptic problems with quadratic growth in the gradient, {\em Comm. Part. Diff. Equ.}, {\bf 38}, (2013), 244--264.

\bibitem{KK}
{\rm J.L.\ Kazdan, R.J.\ Kramer}, Invariant criteria for existence of solutions
to second-order quasilinear elliptic equations, {\em Comm. Pure Appl. Math.},
{\bf 31}, (1978), 619--645.

\bibitem{LiYenKe}
 {\rm J. Li, J. Yin, Y. Ke}, Existence of positive solutions for the p-Laplacian with p-gradient term. {\em J. Math. Anal. Appl.} {\bf 245} (2000),  no. 2, 303-–316.


\bibitem{Lie}  {\rm G. Lieberman},. Boundary regularity for solutions of degenerate elliptic equations. {\em Nonlinear Anal.} 12 (1988), no. 11, 1203-–1219.

\bibitem{Liu}  {\rm S. Liu}, On superlinear problems without the Ambrosetti and Rabinowitz condition. {\em Nonlinear Anal.} {\bf 73} (2010), no. 3, 788-–795.

\bibitem{Mi-So} {\rm O. H. Miyagaki,  M. A. S. Souto,}
Superlinear problems without Ambrosetti and Rabinowitz growth condition. 
{\em J. Differential Equations} {\bf 245} (2008), no. 12, 3628–-3638.

\bibitem{MoMo}
 {\rm M. Montenegro, M. Montenegro}, Existence and nonexistence of solutions for quasilinear elliptic equations. {\em J. Math. Anal. Appl.} {\bf 245} (2000),  no. 2, 303-–316.
 
 \bibitem{MoMu}
 {\rm A. Mokrane, F. Murat},  Lewy-Stampacchia's inequality for a Leray-Lions opeartor with natural growth in the gradient. {\em Boll. Unione Mat. Ital.}  (2014),  no. 7, 55--85.

\bibitem{Mu} {\rm D. Mugnai} Addendum to: Multiplicity of critical points in presence of a linking: application to a superlinear boundary value problem, {\em NoDEA. Nonlinear Differential Equations Appl.} 11 (2004), no. 3, 379-–391, and a comment on the generalized Ambrosetti-Rabinowitz
 condition.
 
 \bibitem{Ru} {\rm  D. Ruiz} A priori estimates and existence of positive solutions for strongly nonlinear problems,  {\em J. Differential Equations} 199 (2004), {\bf no. 1}, 96-–114.
 



\bibitem{Se} 
{\rm J. Serrin},  The solvability of boundary value problems. Mathematical developments arising from Hilbert problems (Proc. Sympos. Pure Math., Northern Illinois Univ., De Kalb, Ill., 1974), pp. 507–-524. {\em Proc. Sympos. Pure Math.}, Vol. XXVIII, Amer. Math. Soc., Providence, R. I., 1976.

\bibitem{PoSe} {\rm  A. Porreta, S. Segura de Le\'on} Nonlinear Elliptic Equations having a gradient term with natural growth,  {\em J. Math. Pures Appl.} 85 (2006), 465-–492.

\bibitem{Va}  {\rm J.L. V\'azquez},  A strong maximum principle for some quasilinear elliptic equations. {\em Appl. Math. Optim.} 12 (1984), no. 3, 191-–202.


\end{thebibliography}
\end{document}